\documentclass[12pt,a4paper]{article}
\usepackage{amsmath}
\usepackage{amssymb}
\usepackage{amssymb}
\usepackage{cases}
\usepackage{bbm}
\usepackage{mathrsfs}
\usepackage{graphicx}
\usepackage{amsfonts}
\usepackage{enumerate}
\usepackage{theorem}
\usepackage[pdfstartview=FitH]{hyperref}
\makeatletter
\def\tank#1{\protected@xdef\@thanks{\@thanks
 \protect\footnotetext[0]{#1}}}
\def\bigfoot{

 \@footnotetext}
\makeatother

\newcommand{\ea}{\end{array}}

\allowdisplaybreaks

\newtheorem{theorem}{Theorem}[section]
\newtheorem{lem}{Lemma}[section]
\newtheorem{prp}[theorem]{Proposition}
\newtheorem{thm}[theorem]{Theorem}
\newtheorem{cor}[theorem]{Corollary}
\newtheorem{dfn}[theorem]{Definition}

\newtheorem{remark}{Remark}

\def\beq{\begin{equation}}
\def\nneq{\end{equation}}

\def\bthm{\begin{thm}}
\def\nthm{\end{thm}}

\def\blem{\begin{lem}}
\def\nlem{\end{lem}}
\def\bprf{\begin{proof}}
\def\nprf{\end{proof}}
\def\bprop{\begin{prop}}
\def\nprop{\end{prop}}
\def\brmk{\begin{rem}}
\def\nrmk{\end{rem}}

\def\bexa{\begin{exa}}
\def\nexa{\end{exa}}
\def\bcor{\begin{cor}}
\def\ncor{\end{cor}}

{\theorembodyfont{\rmfamily}
}

\title{Exponential mixing for stochastic model of two-dimensional second grade fluids
}

\footnotesize{\footnotesize{\author {Ran Wang $^{1,}$\thanks{wangran@ustc.edu.cn},\ \
Jianliang Zhai $^{1,}$\thanks{zhaijl@ustc.edu.cn},\ \
 Tusheng Zhang$^{2,1,}$\thanks{Tusheng.Zhang@manchester.ac.uk}\\
 {\em $^1$ School of Mathematical Sciences,}\\
 {\em University of Science and Technology of China,}\\
 {\em Hefei, 230026, China}\\
 {\em $^2$ School of Mathematics, University of Manchester,}\\
 {\em Oxford Road, Manchester, M13 9PL, UK}\\
}
\date{}
\newenvironment{proof}{\par\noindent{\bf Proof:}}{\hspace*{\fill}$\blacksquare$\par}
\begin{document}
\maketitle

\noindent{\bf Abstract:}
In this paper, we establish the exponential mixing property of stochastic models for the incompressible second grade
fluid. The general criterion established by Cyril Odasso \cite{Odasso} plays an important role.

\vspace{4mm}


\vspace{3mm}
\section{Introduction}

In this paper, we are concerned with the exponential mixing property of stochastic models for the incompressible second grade
fluid which is a particular class of Non-Newtonian fluid. Let $\mathcal{O}$ be a connected, bounded open
subset of $\mathbb{R}^2$ with boundary $\partial \mathcal{O}$ of class $\mathcal{C}^3$. We consider equation
\begin{eqnarray}\label{01}
&  &d(u -\alpha \triangle u )+\Big(-\nu \triangle u +curl(u -\alpha \triangle u )\times u +\nabla\mathfrak{P}\Big)dt\\\nonumber
& &=
\phi(u)dW,\ \ \ in\ \mathcal{O}\times(0,\infty),
\end{eqnarray}
under the following condition
\begin{eqnarray}\label{02}
\left\{
 \begin{array}{llll}
 & \hbox{${\rm{div}}\ u =0\ \text{in}\ \mathcal{O}\times(0,\infty)$;} \\
 & \hbox{$u =0\ \text{in}\ \partial \mathcal{O}\times[0,\infty)$;} \\
 & \hbox{$u (0)=u_0\ \text{in}\ \mathcal{O}$,}
 \end{array}
\right.
\end{eqnarray}
where $u =(u_1 ,u_2 )$ and $\mathfrak{P} $ represent the random velocity and modified pressure, respectively.
$W$ is a cylindrical Wiener process on a Hilbert space $U$ defined on a complete probability space $(\Omega,\mathcal{F},(\mathcal{F}_t)_{t\geq 0},P)$,
and $\phi(u)dW$ represents the external random force.

The interest in the investigation of the second grade fluids arises from the fact that it is an admissible
model of slow flow fluids, which contains a large class Non-Newtonian fluids such as industrial fluids, slurries, polymer melts, etc..
Furthermore, the second grade fluid has general and pleasant properties such as boundedness, stability, and exponential decay (see \cite{Dunn-Fosdick}).
It also has interesting connections with many other fluid models, see \cite{Iftimie, Busuioc, Busuioc-Ratiu, Shkoller2001, Shkoller1998, Holm-Marsden-Ratiu, Holm-Marsden-Ratiu01} and references therein. For example, it can be taken as a generalization of the Navier-Stokes Equation.
Indeed equation (\ref{01}) reduces to Navier-Stokes equation when $\alpha=0$. Furthermore, it was shown in \cite{Iftimie} that the second grade fluids models are good approximations of the Navier-Stokes equation.
Finally we refer to \cite{Noll-truesdell, Dunn-Fosdick, Dunn-Rajagopal, Fosdick-Rajagopal} for a comprehensive theory of the second grade fluids.

The stochastic model of two-dimensional second grade fluids (\ref{01}) has been recently studied in \cite{RS-10}, \cite{RS-12} and \cite{RS-10-01}, where the authors obtained the existence and uniqueness of solutions
and investigated the behavior of the solution as $\alpha\rightarrow0$. We mention that the martingale solution of
the system (\ref{01}) driven by L\'evy noise are studied in \cite{HRS}.

\vskip 0.3cm
In this paper, we establish the exponential mixing property of stochastic models for the incompressible second grade
fluid driven by multiplicative, but possibly degenerate noise. The exponential mixing characterizes the long time behaviour
of the solutions of the stochastic partial differential equations. More precisely, under reasonable conditions, we showed
the equation (\ref{01}) has a unique invariant measure, and the law of the solution converges to the invariant measure
exponentially fast. We will apply the criterion established in \cite{Odasso} by Cyril Odasso.
To this end, we need to prove the exponential integrability of certain energy functionals of the solutions, which is non-trival.

%a non additive noise, and the noise is allowed to be degenerate. The general criterion established  plays an important role.

This article is divided into four sections. In Section 2, we present some preliminaries. Section 3 is devoted to the formulation of the main result.
 The proof of our main result is given in Section 4.

\section{Preliminaries}

In this section, we will introduce some functional spaces and preliminary facts which will be used later.

Let $1\leq p<\infty$, and let $k$ be a nonnegative integer. We denote by $L^p(\mathcal{O})$
and $W^{k,p}(\mathcal{O})$ the usual $L^p$ and Sobolev spaces, and write $W^{k,2}(\mathcal{O})=H^k(\mathcal{O})$. Let $W^{k,p}_0(\mathcal{O})$ be
the closure in $W^{k,p}(\mathcal{O})$ of $\mathcal{C}^\infty_c(\mathcal{O})$ the space of infinitely differentiable functions with compact support in
$\mathcal{O}$. We denote $W^{k,2}_0(\mathcal{O})$ by $H_0^k(\mathcal{O})$. We endow the Hilbert space $H^1_0(\mathcal{O})$
with the scalar product
\begin{eqnarray}\label{H-01}
((u,v))=\int_\mathcal{O}\nabla u\cdot\nabla vdx=\sum_{i=1}^2\int_\mathcal{O}\frac{\partial u}{\partial x_i}\frac{\partial v}{\partial x_i}dx,
\end{eqnarray}
where $\nabla$ is the gradient operator. The norm $\|\cdot\|$ generated by this scalar product is equivalent to the usual norm of $W^{1,2}(\mathcal{O})$
in $H^1_0(\mathcal{O})$.

In what follows, we denote by $\mathbb{X}$ the space of $\mathbb{R}^2$-valued functions such that each component belongs to $X$. We introduce the spaces
\begin{eqnarray}\label{SP-01}
\mathcal{C}=\Big\{u\in[\mathcal{C}^\infty_c(\mathcal{O})]^2\ {\rm such \ that\  div}\ u=0\Big\},\nonumber\\
\mathbb{V}={\rm\ closure\ of}\ \mathcal{C} {\rm\ in}\ \mathbb{H}^1(\mathcal{O}),\\
\mathbb{H}={\rm\ closure\ of}\ \mathcal{C}\ {\rm in}\ \mathbb{L}^2(\mathcal{O}).\nonumber
\end{eqnarray}
We denote by $(\cdot,\cdot)$ and $|\cdot|$ the inner product and the norm induced by the inner product and the norm
in $\mathbb{L}^2(\mathcal{O})$ on $\mathbb{H}$, respectively. The inner product and the norm of $\mathbb{H}^1_0(\mathcal{O})$
are denoted respectively by $((\cdot,\cdot))$ and $\|\cdot\|$. We endow the space $\mathbb{V}$ with the norm generated
by the following scalar product
$$
(u,v)_\mathbb{V}=(u,v)+\alpha ((u,v)),\ \text{for any } v\in\mathbb{V};
$$
which is equivalent to $\|\cdot\|$, more precisely, we have
\begin{eqnarray}\label{Eq 01}
(\mathcal{P}^2+\alpha)^{-1}\|v\|^2_\mathbb{V}
\leq
\|v\|^2
\leq
\alpha^{-1}\|v\|^2_\mathbb{V},\ \ for\ any\ v\in\mathbb{V},
\end{eqnarray}
where $\mathcal{P}$ is the constant from Poincar\'e's inequality.

We also introduce the following space
\begin{eqnarray*}
\mathbb{W}=\{u\in\mathbb{V}\ \text{such that }curl (u-\alpha\triangle u)\in L^2(\mathcal{O})\},
\end{eqnarray*}
and endow it with the norm generated by the scalar product
\begin{eqnarray}\label{W}
(u,v)_\mathbb{W}=(u,v)_\mathbb{V}+\Big(curl(u-\alpha\triangle u),curl(v-\alpha\triangle v)\Big).
\end{eqnarray}
The following result states that the norm induced by $(\cdot,\cdot)_\mathbb{W}$ is equivalent to the usual $\mathbb{H}^3(\mathcal{O})$-norm on $\mathbb{W}$. This result can be found
in \cite{CE}, \cite{CG} and Lemma 2.1 in \cite{RS-12}.

\begin{lem}%\cite{RS-12}
Set
%\begin{eqnarray}\label{W-01}
$
\widetilde{\mathbb{W}}=\Big\{v\in\mathbb{H}^3(\mathcal{O})\text{ such that }{\rm div} v=0\ and\ v|_{\partial \mathcal{O}}=0\Big\}.
$
%\end{eqnarray}
Then the following (algebraic and topological) identity holds:
\begin{eqnarray}\label{W = W}
\mathbb{W}=\widetilde{\mathbb{W}}.
\end{eqnarray}
Moreover, there is a positive constant $C$ such that
\begin{eqnarray}\label{W-02}
    \|v\|^2_{\mathbb{H}^3(\mathcal{O})}
\leq
    C\Big(\|v\|^2_\mathbb{V}+|curl(v-\alpha \triangle v)|^2\Big),
\end{eqnarray}
for any $v\in \widetilde{\mathbb{W}}$.
\end{lem}

%By this lemma we can endow the space $\mathbb{W}$ with norm $|\cdot|_\mathbb{W}$ which is generated by the scalar product
%(\ref{W}).

From now on, we identify the space $\mathbb{V}$ with its dual space $\mathbb{V}^*$ via the Riesz representation, and we have the
Gelfand triple
\begin{eqnarray}\label{Gelfand}
\mathbb{W}\subset \mathbb{V}\subset\mathbb{W}^*.
\end{eqnarray}
 We denote by $\langle f,v\rangle$ the
action of any element $f$ of $\mathbb{W}^*$ on an element $v\in\mathbb{W}$. It is easy to see
$$(v,w)_\mathbb{V}=\langle v,w\rangle,\ \ \ \forall v\in\mathbb{V},\ \ \forall w\in\mathbb{W}.$$

Note that the injection of $\mathbb{W}$ into $\mathbb{V}$ is compact. Thus,
there exists a sequence $\{e_i:i=1,2,3,\cdots\}$ of elements of $\mathbb{W}$ which forms an orthonormal basis in $\mathbb{W}$.
The elements of this sequence are the solutions of the eigenvalue problem
\begin{eqnarray}\label{Basis}
(v,e_i)_{\mathbb{W}}=\lambda_i(v,e_i)_{\mathbb{V}},\ \text{for any }v\in\mathbb{W}.
\end{eqnarray}
Here $\{\lambda_i:i=1,2,3,\cdots\}$ is an increasing sequence of positive eigenvalues. We have the following important result from
\cite{CG} about the regularity of the functions $e_i,\ i=1,2,3,\cdots.$

\begin{lem}\label{lem Basis}
Let $\mathcal{O}$ be a bounded, simply-connected open subset of $\mathbb{R}^2$ with a boundary of class $\mathcal{C}^3$, then
the eigenfunctions of (\ref{Basis}) belong to $\mathbb{H}^4(\mathcal{O})$.
\end{lem}

 Consider the following ``generalized Stokes equations":

\begin{eqnarray}\label{General Stokes}
v-\alpha \triangle v=f\ {\rm in}\ \mathcal{O},\nonumber\\
{\rm div}\ v=0\ {\rm in}\ \mathcal{O},\\
v=0\ {\rm on}\ \partial \mathcal{O}.\nonumber
\end{eqnarray}
%Let us stating an important lemma on solvability of (\ref{General Stokes}).
%Its proof can be derived from an adaptation of the results obtained by Solonnikov in \cite{SO1}, \cite{SO2},
%and also can be found in Theorem 2.5 of \cite{RS-10} and Theorem 2.2 of \cite{RS-12}.

The following result can be found in \cite{SO1}, \cite{SO2}, Theorem 2.5 of \cite{RS-10} and Theorem 2.2 of \cite{RS-12}.
\begin{lem}\label{Lem GS}
Let $\mathcal{O}$ be a connected, bounded open subset of $\mathbb{R}^2$ with boundary $\partial \mathcal{O}$ of class $\mathcal{C}^l$
and let $f$ be a function in $\mathbb{H}^l$, $l\geq 1$. Then the system (\ref{General Stokes}) admits a solution $v\in \mathbb{H}^{l+2}\cap\mathbb{V}$.
Moreover if $f$ is an element of $\mathbb{H}$, then $v$ is unique and the following relations hold
\begin{eqnarray}\label{Eq GS-01}
(v,g)_\mathbb{V}=(f,g),\ \text{for any }g\in \mathbb{V},
\end{eqnarray}
and
\begin{eqnarray}\label{Eq GS-02}
\|v\|_{\mathbb{W}}\leq K\|f\|_\mathbb{V}.
\end{eqnarray}

\end{lem}

Define the Stokes operator by
\begin{eqnarray}\label{Eq Stoke}
Au=-\mathbb{P}\triangle u,\ \forall u\in D(A)=\mathbb{H}^2(\mathcal{O})\cap\mathbb{V},
\end{eqnarray}
here we denote by $\mathbb{P}:\mathbb{L}^2(\mathcal{O})\rightarrow\mathbb{H}$ the usual Helmholtz-Leray projector.
It follows from Lemma \ref{Lem GS} that the operator $(I+\alpha A)^{-1}$ defines an isomorphism from $\mathbb{H}^l(\mathcal{O})\cap\mathbb{H}$
into $\mathbb{H}^{l+2}(\mathcal{O})\cap\mathbb{V}$ provided that $\mathcal{O}$ is of class $\mathcal{C}^l$, $l\geq1$. Moreover, the following
properties hold
\begin{eqnarray*}
& & ((I+\alpha A)^{-1}f,g)_\mathbb{V}=(f,g),\\
& & \|(I+\alpha A)^{-1}f\|_\mathbb{W}\leq K\|f\|_{\mathbb{V}},
\end{eqnarray*}
for any $f\in \mathbb{H}^l(\mathcal{O})\cap\mathbb{V}$ and any $g\in\mathbb{V}$. From these facts, $\widehat{A}=(I+\alpha A)^{-1}A$ defines a
continuous linear operator from $\mathbb{H}^l(\mathcal{O})\cap\mathbb{V}$ onto itself for $l\geq2$, and satisfies
$$
(\widehat{A}u,g)_\mathbb{V}=(Au,g)=((u,g)),
$$
for any $u\in\mathbb{W}$ and $g\in\mathbb{V}$. Hence, for any $u\in\mathbb{W}$
$$
(\widehat{A}u,u)_\mathbb{V}=\|u\|.
$$

Let
$$
b(u,v,w)=\sum_{i,j=1}^2\int_{\mathcal{O}}u_i\frac{\partial v_j}{\partial x_i}w_jdx,
$$
for any $u,v,w\in\mathcal{C}$. Then the following identity holds(see for instance \cite{CO} \cite{Bernard}):
\begin{eqnarray}\label{curl}
((curl\Phi)\times v,w)=b(v,\Phi,w)-b(w,\Phi,v),
\end{eqnarray}
for any smooth function $\Phi,\ v$ and $w$. Now we recall the following two lemmas which can be found in \cite{RS-12}(Lemma 2.3 and Lemma 2.4),
 and also in \cite{CO} \cite{Bernard}.

\begin{lem}\label{Lem B}
For any $u,v,w\in\mathbb{W}$, we have
\begin{eqnarray}\label{Ineq B 01}
    |(curl(u-\alpha\Delta u)\times v,w)|
\leq
    \widetilde{K}\|u\|_{\mathbb{H}^3}\|v\|_\mathbb{V}\|w\|_{\mathbb{W}},
\end{eqnarray}
and
\begin{eqnarray}\label{Ineq B 02}
    |(curl(u-\alpha\Delta u)\times u,w)|
\leq
    \Theta\|u\|^2_\mathbb{V}\|w\|_{\mathbb{W}}.
\end{eqnarray}
\end{lem}

Set $$B(u,v)=curl(u-\alpha\Delta u)\times v, \ \ \forall u,v\in\mathbb{W}.$$
%and define the bilinear operator $\widehat{B}(\cdot,\cdot):\ \mathbb{W}\times\mathbb{V}\rightarrow\mathbb{W}^*$ as
%\begin{eqnarray}\label{Lem B-01}
%\widehat{B}(u,v)=(I+\alpha A)^{-1}B(u,v).
%\end{eqnarray}
%\begin{lem}\label{Lem-B-01}
%For any $u\in\mathbb{W}$ and $v\in\mathbb{V}$ there holds
%\begin{eqnarray}\label{Eq B-01}
%    \|\widehat{B}(u,v)\|_{\mathbb{W}^*}
%\leq
%    C\|u\|_\mathbb{W}\|v\|_\mathbb{V},
%\end{eqnarray}
%and
%\begin{eqnarray}\label{Eq B-02}
% \|\widehat{B}(u,u)\|_{\mathbb{W}^*}
%\leq
%    C_B\|u\|^2_\mathbb{V}.
%\end{eqnarray}
%In addition
%\begin{eqnarray}\label{Eq B-03}
% \langle\widehat{B}(u,v),v\rangle=0,
%\end{eqnarray}
%which implies
%\begin{eqnarray}\label{Eq B-04}
% \langle\widehat{B}(u,v),w\rangle=-\langle\widehat{B}(u,w),v\rangle,
%\end{eqnarray}
%for any $u,\ v,\ w\in\mathbb{W}$.
%
%\end{lem}

\section{Formulation of the main result}

In this section, we will state the precise assumptions on the coefficients and collect some preliminary results from \cite{RS-10} and \cite{RS-12}, which will be used in the following sections.

 Assume that $\{W(s),\ s\in[0,\infty)\}$ is a $U$-cylindrical Wiener process admitting the following representation:
  $$
  W=\sum_{n}\beta_n\bar{e}_n,
  $$
  where $(\bar{e}_n)_n$ is a complete orthonormal system of $U$ and $(\beta_n)_n$ is a sequence of independent Brownian motions.

  Given Hilbert spaces $Q_1,Q_2$, we denote by $\mathcal{L}_2(Q_1,Q_2)$ the space of all Hilbert-Schmidt operators from $Q_1$ into $Q_2$.
  And $\mathcal{L}(Q_1,Q_2)$ denotes the space of bounded linear operators from $Q_1$ into $Q_2$.
Let $\phi:\mathbb{V}\rightarrow \mathcal{L}_2(U,\mathbb{V})$ be a given measurable mapping. We denote by $P_N$
the orthogonal projection from $\mathbb{V}$ into the space $Span(e_1,\cdots,e_N)$. Now we introduce the following conditions:

{\bf (H0)} The mapping $\phi:\mathbb{V}\rightarrow\mathcal{L}_2(U,\mathbb{V})$ is bounded and Lipschitz, i.e., there exist constants
$R$ and $L_\phi$ such that
$$
R=\sup_{v\in\mathbb{V}}\|\phi(v)\|^2_{\mathcal{L}_2(U,\mathbb{V})},
$$
and
$$
\|\phi(v_1)-\phi(v_2)\|_{\mathcal{L}_2(U,\mathbb{V})}\leq L_\phi\|v_1-v_2\|_\mathbb{V},\ \ \ \forall v_1,v_2\in\mathbb{V}.
$$

{\bf (H1)} Recall the constant $K$ in Lemma \ref{Lem GS}, and $\Theta$ in (\ref{Ineq B 02}).
There exists $N\in\mathbb{N}$ and a bounded measurable mapping $g:\mathbb{V}\rightarrow\mathcal{L}(\mathbb{V},U)$ such
that for any $v\in\mathbb{V}$
\begin{eqnarray}\label{eq c N}
\phi(v)g(v)=P_N,
\end{eqnarray}
and the viscosity constant $\nu$ satisfies
\begin{eqnarray}\label{eq c nu}
\frac{1}{2\Theta^2}\frac{\nu}{\mathcal{P}^2+\alpha}\Big(\frac{2\nu}{\mathcal{P}^2+\alpha}-1-\frac{K^2L^2_{\phi}}{\lambda_1}\Big)
\geq
\Big(1+\frac{2}{\lambda_1}+\frac{2(\mathcal{P}^2+\alpha)}{\lambda_1\alpha^2}\Big)K^2R.
\end{eqnarray}

\begin{remark}
(\ref{eq c N}) can be seen as a non degeneracy condition on the low modes, and (\ref{eq c nu}) is a technique condition.
%is equivalent to the following property
%$$P_N\mathbb{V}\subset Im(\phi(v)).$$
\end{remark}

%For for any $u_1,u_2\in\mathbb{V}$,
%\begin{eqnarray}\label{G-01}
%\phi(0,t)=0,
%\end{eqnarray}
%and
%\begin{eqnarray}\label{G-02}
%\|G(u_1,t)-G(u_2,t)\|_{\mathbb{V}^{\otimes m}}
%\leq
%C\|u_1-u_2\|_\mathbb{V}.
%\end{eqnarray}

%We now define the operator $\widehat{\phi}$ which map $\mathbb{V}$ into $\mathcal{L}_2(U,\mathbb{W})$ by
%\begin{eqnarray*}
%\widehat{\phi}(u)=(I+\alpha A)^{-1} \phi(u).
%\end{eqnarray*}
%
%{\bf (H0)} implies that there exist $\widehat{R}$, $\widehat{L_\phi}$
%such that
%$$
%\widehat{R}=\sup_{v\in\mathbb{V}}\|\widehat{\phi}(v)\|^2_{\mathcal{L}_2(U,\mathbb{V})}
%$$
%and
%$$
%\|\widehat{\phi}(v_1)-\widehat{\phi}(v_2)\|_{\mathcal{L}_2(U,\mathbb{V})}\leq \widehat{L_\phi}\|v_1-v_2\|_\mathbb{V},\ \ \ \forall v_1,v_2\in\mathbb{V}.
%$$
%
%By applying $(I+\alpha A)^{-1}$ to (\ref{01}), we obtain the abstract stochastic evolution equations
%\begin{eqnarray}\label{Abstract}
%du (t)+\nu \widehat{A}u (t)dt+\widehat{B}(u (t),u (t))dt=\widehat{\phi}(u (t))dW(t),
%\end{eqnarray}
%with initial value $u_0=u(0)$, which holds in $\mathbb{W}^*$. With the properties of the operators involved, it can be proved that a stochastic
%process $u $ satisfies (\ref{Abstract}) if and only if it verifies (\ref{01}) in the weak sense of partial differential equations.
%

Now we recall the concept of solution of the problem (\ref{01}) in \cite{RS-12}.

\begin{dfn}\label{Def 01}
A stochastic process $u $ is called a solution of the system (\ref{01}), if

1. $u(0)=u_0$,

2. $u \in L^p(\Omega,\mathcal{F},P;L^\infty([0,\infty),\mathbb{W})),\ 2\leq p<\infty,$

3. For all $t\geq 0$, $u (t)$ is $\mathcal{F}_t$-measurable,

4. For any $t\in(0,\infty)$ and $v\in\mathbb{W}$, the following identity holds almost surely
\begin{eqnarray*}
&  &(u (t)-u (0),v)_{\mathbb{V}}+\int_0^t[\nu ((u (s),v))+(curl (u (s)-\alpha \Delta u (s))\times u (s),v)]ds\\
&=&
    \int_0^t(\phi(u (s))dW(s),v).
\end{eqnarray*}
%Or equivalently, the following equation
%\begin{eqnarray*}
%u (t)+\int_0^t\Big(\nu \widehat{A}u (s)+\widehat{B}(u (s),u (s))\Big)ds
%=
%u_0+\int_0^t\widehat{\phi}(u (s))dW(s),
%\end{eqnarray*}
%holds in $\mathbb{W}^*$ $P$-a.s..
\end{dfn}

Using Galerkin approximation scheme for the system (\ref{01}), Razafimandimby and Sango \cite{RS-12}
obtained the following theorem (see Theorem 3.4 and Theorem 4.1 in \cite{RS-12}).
\begin{thm}\label{Solution Existence}
Let $u_0\in\mathbb{W}$. Assume {\bf (H0)} holds. Then
\begin{itemize}
  \item[(1)] the system (\ref{01}) has a unique solution,

  \item[(2)] the solution $u$ admits a version which is continuous in $\mathbb{V}$ with respect to
the strong topology and continuous in $\mathbb{W}$ with respect to
the weak topology.
 \end{itemize}
 \end{thm}
 Moreover, from the proof of Theorem 4.1 in \cite{RS-12}, we have
 \begin{thm}\label{Solution Existence 01}
  Assume that $u_1,u_2\in\mathbb{W}$ are $\mathcal{F}_t$-measurable, and let $\{X_1(t+s),\ s\geq0\}$ and $\{X_2(t+s),\ s\geq0\}$ be two solutions of the system (\ref{01}) with initial condition $X_1(t)=u_1$
 and $X_2(t)=u_2$, respectively. Then, for any $O\in\mathcal{F}_t$
 \begin{eqnarray}\label{eq:01}
 & &E\Big(\sigma(t+s,t)\|X_1(t+s)-X_2(t+s)\|^2_\mathbb{V}1_{O}\Big)\nonumber\\
 &\leq&
 E\Big(\|u_1-u_2\|^2_\mathbb{V}1_{O}\Big)+C\int_t^{t+s}E\Big(\sigma(l,t)\|X_1(l)-X_2(l)\|^2_\mathbb{V}1_{O}\Big)dl,
 \end{eqnarray}
here $\sigma(l,t)=\exp\Big(-\int_t^{l}\|X_2(s)\|^2_\mathbb{W}ds\Big)$.
\end{thm}

\begin{remark}\label{Remark 01}
By (\ref{eq:01}), if $u_1=u_2$ on $O\in\mathcal{F}_t$, then $X_1(t+\cdot)=X_2(t+\cdot)$ on $O$ $P$-a.s..
\end{remark}

For a $\mathbb{W}$-valued, $\mathcal{F}_{t_0}$-measurable random variable $Y$, let $u(t_0+\cdot,t_0,Y)$ be the unique solution of (\ref{01}) on the time interval $[0,\infty)$ with
 initial condition $u(t_0,t_0,Y)=Y$. Denote
\begin{eqnarray}\label{X}
X^x(t)=X(t,W,x)=
\left\{
 \begin{array}{ll}
 & \hbox{$u(t,0,x),\ \ \ x\in\mathbb{W}$;} \\
 & \hbox{$x,\ \ \ x\in\mathbb{V}/\mathbb{W}$.}
 \end{array}
\right.
\end{eqnarray}
Then we define the operators $\mathcal{P}_t:\ B_b(\mathbb{V})\rightarrow B_b(\mathbb{V})$ as
$$
(\mathcal{P}_t\varphi)(x)=E[\varphi(X^x(t))],
$$
where $B_b(\mathbb{V})$ is the space of bounded measurable functions on $\mathbb{V}$. Let $C_b(\mathbb{V})$ be the space of bounded continuous
 functions.

\begin{lem}
$\{X^x,\ x\in\mathbb{V}\}$ defines a Markov process in the sense that, for every $x\in\mathbb{V},\ \varphi\in C_b(\mathbb{V}),\ t,s>0$
\begin{eqnarray}\label{Eq semi 01}
E[\varphi(X^x(t+s))|\mathcal{F}_t]=(\mathcal{P}_s\varphi)(X^x(t)),\ \ P\text{-}a.s..
\end{eqnarray}
\end{lem}
\begin{proof}
Noticing that (\ref{Eq semi 01}) holds when $x\in\mathbb{V}/\mathbb{W}$.
Now given $x\in\mathbb{W}$, we have
$X^x(t)=u(t,0,x)$. To prove (\ref{Eq semi 01}), it is sufficient to prove that
\begin{eqnarray*}
E[\varphi(u(t+s,0,x))Z]=E[(\mathcal{P}_s\varphi)(u(t,0,x))Z]
\end{eqnarray*}
for every bounded $\mathcal{F}_t$-measurable r.v. Z.

Since, by Theorem \ref{Solution Existence},
$$
u(t+s,0,x)=u(s,t,u(t,0,x))\ \ \ and\ \ \ E(\|u(t,0,x)\|^2_\mathbb{W})<\infty,
$$
it is sufficient to prove that
\begin{eqnarray}\label{Eq:02}
E[\varphi(u(t+s,t,\eta))Z]=E[(\mathcal{P}_s\varphi)(\eta)Z]
\end{eqnarray}
for every $\mathbb{W}$-valued $\mathcal{F}_t$-measurable r.v. $\eta$.

By (\ref{eq:01}), for any given $\xi_n,\xi\in\mathbb{W}$, the strong convergence of $\xi_n$ to $\xi$ in $\mathbb{V}$
implies that $(\mathcal{P}_s\varphi)(\xi_n)$ converges to $(\mathcal{P}_s\varphi)(\xi)$. Hence, to prove (\ref{Eq:02}), it is sufficient to prove it
for every r.v. $\eta$ of the form $\eta=\sum_{i=1}^k \eta^i1_{A^i}$ with $\eta^i\in\mathbb{W}$ and $A^i\in\mathcal{F}_t$.
By Remark \ref{Remark 01}, we just need to prove (\ref{Eq:02}) for every deterministic $\eta\in\mathbb{W}$.

Now the r.v. $u(t+s,t,\eta)$ depends only on the increments of the Brownian motion between $t$ and $t+s$, hence it is independent
of $\mathcal{F}_t$. Therefore
$$
E[\varphi(u(t+s,t,\eta))Z]=E[\varphi(u(t+s,t,\eta))]EZ.
$$
Since $u(t+s,t,\eta)$ has the same law of $u(s,0,\eta)$(by uniqueness), we have $E[\varphi(u(t+s,t,\eta))]=E[\varphi(u(s,0,\eta))]$
and thus
$$
E[\varphi(u(t+s,t,\eta))Z]=E[\varphi(u(s,0,\eta))]EZ=E[\varphi(u(s,0,\eta))Z].
$$
The proof is complete.

\end{proof}

The space of probability measures on $\mathbb{V}$ is denoted by $\mathcal{P}(\mathbb{V})$. The aim of this paper is to prove the following result.
\begin{thm}\label{thm main}
Assume that \textbf{(H0)} holds, and \textbf{(H1)} holds with some $N\in\mathbb{N}$.
Then there exits an unique invariant probability measure $\mu$ of $(\mathcal{P}_t)_{t\in\mathbb{R}^+}$ on $\mathbb{V}$
satisfying
$$
\int_{\mathbb{V}}\|u\|^2_{\mathbb{W}}\mu(du)<\infty,
$$
and there exist $C,\gamma'>0$ such that for any $\lambda\in\mathcal{P}({\mathbb{V}})$
\begin{eqnarray*}
\|\mathcal{P}^*_t\lambda-\mu\|_*\leq Ce^{-\gamma' t}\Big(1+\int_{\mathbb{V}}\|u\|^2_\mathbb{W}\lambda(du)\Big).
\end{eqnarray*}

\end{thm}

\section{Proof of the main result}

This section is devoted to the proof of the main result.
We first recall the general criterion established in \cite{Odasso}.

Given a Polish space $E$, $Lip_b(E)$ will denote the space of all bounded, Lipschitz continuous functions on $E$. Set
$$
\|\varphi\|_L=|\varphi|_\infty+L_\varphi,\ \ \varphi\in Lip_b(E),
$$
here $|\cdot|_\infty$ is the sup norm and $L_\varphi$ is the Lipschitz constant of $\varphi$. The space of probability meaures on $E$
is denoted by $\mathcal{P}(E)$. It is endowed with the Wasserstein norm
$$
\|\mu\|_*=\sup_{\varphi\in Lip_b(E),\|\varphi\|_L\leq 1}|\int_E\varphi(u)\mu(du)|,\ \ \mu\in \mathcal{P}(E).
$$

Let $(U,|\cdot|_U)$ and $(\mathbb{V},\|\cdot\|_\mathbb{V})$ be the two Hilbert spaces introduced before. We consider a Markov process $\Upsilon$ living in $\mathbb{V}$
and depending measurably on a cylindrical Wiener process $W$ on $U$. $\Upsilon$ can be written as
$$
\Upsilon(t)=\Upsilon(t,W,x_0),
$$
where $x_0$ is the initial value $\Upsilon(0,W,x_0)=x_0$. We denote the distribution of $\Upsilon(\cdot,W,x_0)$ by $\mathcal{D}(\Upsilon(\cdot,W,x_0))$,
and assume that $\mathcal{D}(\Upsilon(\cdot,W,x_0))$ is measurable with respect to $x_0$. Let $(\mathcal{P}_t)_{t\geq 0}$ be the Markov
transition semigroup associated with the Markov family $(\Upsilon(\cdot,W,x_0))_{x_0\in {\mathbb{V}}}$.

The basis idea behind the criterion in \cite{Odasso} is to construct an auxiliary process $\widetilde{\Upsilon}(t,W,x_0,\widetilde{x}_0)$, which is
``close" to $\Upsilon(t,W,x_0)$ and has a law absolutely continuous with respect to $\mathcal{D}(\Upsilon(\cdot,W,\widetilde{x}_0))$. More
precisely, suppose that there exists a function
$$
\widetilde{\Upsilon}:\ [0,\infty)\times C([0,\infty); \mathbb{R})^\mathbb{N}\times {\mathbb{V}}\times {\mathbb{V}}\rightarrow {\mathbb{V}},
$$
satisfying the following conditions.

\begin{itemize}
\item[\textbf{(A)}] For every $x_0,\widetilde{x}_0\in {\mathbb{V}}$, $\widetilde{\Upsilon}(\cdot,W,x_0,\widetilde{x}_0)$ is non-anticipative and measurable with respect to $W$.
Moreover,
$$
(\Upsilon(t),\widetilde{\Upsilon}(t))=(\Upsilon(t,W,x_0),\widetilde{\Upsilon}(t,W,x_0,\widetilde{x}_0))
$$
defines an homogenous Markov process and its law $\mathcal{D}(\Upsilon,\widetilde{\Upsilon})$ is measurable with respect to $(x_0,\widetilde{x}_0)$.

\item[\textbf{(B)}] There exists a positive measurable function $\mathcal{H}:{\mathbb{V}}\rightarrow\mathbb{R}^+$ and a positive constant $\gamma$ such that
for any $x_0\in {\mathbb{V}}$, $t\geq 0$, $\beta>0$ and any stopping time
$\tau\geq 0$, there exists $C_1, C'_\beta>0$ satisfying
\begin{eqnarray*}
\left\{
 \begin{array}{ll}
 & \hbox{$E\Big(\mathcal{H}(\Upsilon(t,W,x_0))\Big)\leq e^{-\gamma t}\mathcal{H}(x_0)+C_1$;} \\
 & \hbox{$E\Big(e^{-\beta \tau}\mathcal{H}(\Upsilon(\tau,W,x_0))1_{\tau<\infty}\Big)\leq\mathcal{H}(x_0)+C'_\beta$;}
 \end{array}
\right.
\end{eqnarray*}

\item[\textbf{(C)}] There exists a function $h:{\mathbb{V}}\times {\mathbb{V}}\rightarrow U$ such that for any $(t,x^1_0,x^2_0)\in [0,\infty)\times {\mathbb{V}}\times {\mathbb{V}}$
and  cylindrical Wiener process $W$ on $U$, we have almost surely
$$
\widetilde{\Upsilon}(t,W,x_0^1,x_0^2)=\Upsilon\Big(t, W+\int_0^\cdot h\Big(\Upsilon(s,W,x_0^1),\widetilde{\Upsilon}(s,W,x_0^1,x_0^2)\Big)ds,x_0^2\Big);
$$

\item[\textbf{(D)}] For any $x^1_0,x_0^2\in {\mathbb{V}}$ satisfying
\begin{eqnarray}\label{Eq: H}
\mathcal{H}(x_0^1)+\mathcal{H}(x_0^2)\leq 2C_1,
\end{eqnarray}
and for any  cylindrical Wiener processes $W_1,W_2$ on $U$, set $$h(t)=h(\Upsilon(t,W_1,x_0^1),\widetilde{\Upsilon}(t,W_1,x_0^1,x_0^2)),$$
there exists $\gamma_0>0$ such that
\begin{itemize}
\item[\textbf{(D1)}] there exists $C>0$ such that
\begin{eqnarray*}
& &\mathbb{P}\Big(|\Upsilon(t,W_2,x_0^2)-\Upsilon(t,W_1,x_0^1)|_{\mathbb{V}}\geq Ce^{-{\gamma_0} t}, \widetilde{\Upsilon}(\cdot,W_1,x_0^1,x_0^2)=\Upsilon(\cdot,W_2,x_0^2)\ on\ [0,t]\Big)\\
&&\leq
Ce^{-{\gamma_0} t},\ \ \forall t\geq 0;
\end{eqnarray*}

\item[\textbf{(D2)}] for any $t_0\geq 0$ and any stopping time $\tau\geq t_0$, we have
\begin{eqnarray*}
&&\mathbb{P}\Big(\int_{t_0}^\tau|h(t)|^2_Udt\geq Ce^{-{\gamma_0} t_0}\ and\ \widetilde{\Upsilon}(\cdot,W_1,x_0^1,x_0^2)=\Upsilon(\cdot,W_2,x_0^2)\ on\ [0,\tau]\Big)\\
&&
\leq
Ce^{-{\gamma_0} t_0};
\end{eqnarray*}

\item[\textbf{(D3)}] there exists $p_1>0$ such that
$$
\mathbb{P}\Big(\int_{0}^{+\infty}|h(t)|^2_Udt\leq C\Big)\geq p_1.
$$
\end{itemize}

\end{itemize}

Here is the criteria obtained \cite{Odasso}.
\begin{thm}\label{Thm biaozhun}
Under the assumptions \textbf{(A)}--\textbf{(D)}, there exists a unique stationary probability measure $\mu$ of $(\mathcal{P}_t)_{t\in\mathbb{R}^+}$
on ${\mathbb{V}}$, satisfying
\begin{eqnarray*}
\int_{\mathbb{V}}\mathcal{H}(u)d\mu(u)<\infty,
\end{eqnarray*}
and there exist $C,\gamma'>0$ such that for any $\lambda\in\mathcal{P}({\mathbb{V}})$
\begin{eqnarray*}
\|\mathcal{P}^*_t\lambda-\mu\|_*\leq Ce^{-\gamma' t}\Big(1+\int_{\mathbb{V}}\mathcal{H}(u)d\lambda(u)\Big).
\end{eqnarray*}
\end{thm}

\subsection{The proof}

As a part of the proof, we will prepare a number of estimates for the solutions of the equation (\ref{01}).

From now on, we denote by $C$ any generic constant which may change from one line to another.

Set $\mathbb{W}_M=Span(e_1,\cdots,e_M)$. Let $u^M\in\mathbb{W}_M$ be the Galerkin approximations of (\ref{01}) satisfying
\begin{eqnarray}\label{Eq Galerkin}
& &d(u^M,e_i)_\mathbb{V}+\nu((u^M,e_i))dt+b(u^M,u^M,e_i)dt-\alpha b(u^M,\triangle u^M,e_i)dt+\alpha b(e_i,\triangle u^M,u^M)dt\nonumber\\
&=&
 (\phi(u^M),e_i)dW(t),\ i\in\{1,2,\cdots,M\},
\end{eqnarray}
where the notation $(\phi(u),e_i)$ stands for the operator in $\mathcal{L}(U,\mathbb{R})$ defined by
$$
(\phi(u),e_i)h=(\phi(u)h,e_i),\ \ \forall h\in U.
$$
Then
\begin{eqnarray*}
\|(\phi(u),e_i)\|^2_{\mathcal{L}_2(U,\mathbb{R})}
&=&
\sum_{j=1}^\infty(\phi(u)\bar{e}_j,e_i)^2
\leq
\|e_i\|^2_\mathbb{H}\sum_{j=1}^\infty\|\phi(u)\bar{e}_j\|^2_\mathbb{H}\\
&\leq&
C\|e_i\|^2_\mathbb{H}\sum_{j=1}^\infty\|\phi(u)\bar{e}_j\|^2_\mathbb{V}
\leq
C\|\phi(u)\|^2_{\mathcal{L}_2(U,\mathbb{V})}.
\end{eqnarray*}

We have the following result for $u^M$.
\begin{lem}\label{Lem 01}
Assume {\bf (H0)} holds. There exist $C_1$ and $C^1_\beta$ only depending on $\nu,\ \alpha,\ R$ and $\mathcal{O}$ such that
, for any $t\geq0$, any $\beta>0$ and any stopping time $\tau$,
\begin{eqnarray}\label{eq 01 lem 01}
E(\|u^M(t)\|^2_\mathbb{W})\leq e^{-\frac{\nu t}{\mathcal{P}^2+\alpha}}\|u^M_0\|^2_\mathbb{W}+C_1
\end{eqnarray}
and
\begin{eqnarray}\label{eq 02 lem 01}
E(e^{-\beta\tau}\|u^M(\tau)\|^2_\mathbb{W}I_{\tau<\infty})\leq \|u^M_0\|^2_\mathbb{W}+C^1_\beta.
\end{eqnarray}

\end{lem}
\begin{proof}
Applying ${\rm It\hat{o}}$'s formula, we have
\begin{eqnarray}\label{eq 011}
& &d(u^M,e_i)^2_{\mathbb{V}}\nonumber\\
& &+
2(u^M,e_i)_{\mathbb{V}}\Big[\nu((u^M,e_i))+b(u^M,u^M,e_i)-\alpha b(u^M,\triangle u^M,e_i)+\alpha b(e_i,\triangle u^M,u^M)\Big]dt\nonumber\\
&=&
2(u^M,e_i)_{\mathbb{V}}(\phi(u^M),e_i)dW(t)+|(\phi(u^M),e_i)|^2_{\mathcal{L}_2(U,\mathbb{R})} dt.
\end{eqnarray}
Notice that $\|u^M\|^2_\mathbb{V}=\sum_{i=1}^M\lambda_i(u^M,e_i)^2_\mathbb{V}$. Multiplying \label{eq 011} by $\lambda_i$ and taking summation over $i$, we get
\begin{eqnarray}\label{Eq G ito}
d\|u^M\|^2_\mathbb{V}+2\nu\|u^M\|^2dt
=
2(\phi(u^M),u^M)dW(t)+\sum_{i=1}^M\lambda_i|(\phi(u^M),e_i)|^2_{\mathcal{L}_2(U,\mathbb{R})} dt
\end{eqnarray}
here we used the fact that $b(u^M,u^M,u^M)=0$.

Let $\widetilde{G}(u^M(t))$ be the operator in $\mathcal{L}_2(U,\mathbb{W})$ defined as follows. For any $h\in U$,
$\widetilde{G}(u^M(t))\times h\in\mathbb{W}$ is the unique solution of the following equation.
\begin{eqnarray*}
\widetilde{G}(u^M(t))\times h-\alpha \triangle \Big(\widetilde{G}(u^M(t))\times h\Big)=\phi(u^M(t))\times h\ {\rm in}\ \mathcal{O},\nonumber\\
%{\rm div}\ \widetilde{G}(u^M)=0\ {\rm in}\ \mathcal{O},\\
\widetilde{G}(u^M(t))\times h=0\ {\rm on}\ \partial \mathcal{O}.\nonumber
\end{eqnarray*}
By Lemma \ref{Lem GS}, such a solution uniquely exits. Moreover,
\begin{eqnarray}\label{eq:G}
(\widetilde{G}(u^M(t))\times h,e_i)_\mathbb{V}=(\phi(u^M(t))\times h,e_i),\ \ \forall i\in\{1,2,\cdots,M\},
\end{eqnarray}
and there exists a positive constant $K$ such that
$$
\|\widetilde{G}(u^M(t))\times h\|_{\mathbb{W}}
\leq
K\|\phi(u^M(t))\times h\|_{\mathbb{V}},
$$
which implies that
$$
\|\widetilde{G}(u^M(t))\|^2_{\mathcal{L}_2(U,\mathbb{W})}
\leq
K^2\|\phi(u^M(t))\|^2_{\mathcal{L}_2(U,\mathbb{V})}.
$$

Hence by (\ref{Basis}),
\begin{eqnarray}\label{Eq ito 04}
& &\sum_{i=1}^M\lambda_i|(\phi(u^M(t)),e_i)|^2_{\mathcal{L}_2(U,\mathbb{R})}\nonumber\\
&=&
\sum_{i=1}^M\lambda_i|(\widetilde{G}(u^M(t)),e_i)_\mathbb{V}|^2_{\mathcal{L}_2(U,\mathbb{R})}\nonumber\\
&=&
\sum_{i=1}^M\frac{1}{\lambda_i}|(\widetilde{G}(u^M(t)),e_i)_\mathbb{W}|^2_{\mathcal{L}_2(U,\mathbb{R})}\nonumber\\
&\leq&
\frac{1}{\lambda_1}\|\widetilde{G}(u^M(t))\|^2_{\mathcal{L}_2(U,\mathbb{W})}\nonumber\\
&\leq&
\frac{K^2}{\lambda_1}\|\phi(u^M(t))\|^2_{\mathcal{L}_2(U,\mathbb{V})}\nonumber\\
&\leq&
\frac{K^2}{\lambda_1}R.
\end{eqnarray}
Combining this with (\ref{Eq G ito}), we obtain
\begin{eqnarray}\label{Eq lem 01}
d\|u^M\|^2_\mathbb{V}+2\nu\|u^M\|^2dt
\leq
2(\phi(u^M),u^M)dW(t)+\frac{K^2}{\lambda_1}Rdt.
\end{eqnarray}

Denote $\|v\|_*=|curl(v-\alpha\triangle v)|$ for any $v\in\mathbb{W}$. Next we estimate $\|u^M\|_*$.

Setting
$$
\Psi(u^M)=-\nu\triangle u^M+curl(u^M-\alpha\triangle u^M)\times u^M,
$$
we have
$$
d(u^M,e_i)_\mathbb{V}+(\Psi(u^M),e_i)dt=(\phi(u^M),e_i)dW(t).
$$

Noting $\Psi(u^M)\in\mathbb{H}^1(\mathcal{O})$. By Lemma \ref{Lem GS}, there exists unique solution $v^M\in\mathbb{W}$ satisfying

\begin{eqnarray*}
\left\{
 \begin{array}{ll}
 & \hbox{$v^M-\alpha \triangle v^M=\Psi(u^M)\ {\rm in}\ \mathcal{O}$;} \\
 & \hbox{$v^M=0\ {\rm on}\ \partial \mathcal{O}.$}
 \end{array}
\right.
\end{eqnarray*}
Moreover,
$$
(v^M,e_i)_\mathbb{V}=(\Psi(u^M),e_i),\ \ \forall i\in\{1,2,\cdots,M\}.
$$
Thus
\begin{eqnarray}\label{eq 02}
d(u^M,e_i)_\mathbb{V}+(v^M,e_i)_\mathbb{V}dt= (\phi(u^M),e_i)dW(t).
\end{eqnarray}

By (\ref{eq:G}) and (\ref{Basis}), we have
$$
\lambda_i(\phi(u^M),e_i)=(\widetilde{G}(u^M),e_i)_{\mathbb{W}}.
$$
Multiplying $\lambda_i$ to (\ref{eq 02}), it follows that
$$
d(u^M,e_i)_\mathbb{W}+(v^M,e_i)_\mathbb{W}dt= (\widetilde{G}(u^M),e_i)_{\mathbb{W}}dW(t).
$$
Applying ${\rm It\hat{o}}$'s formula, we have
\begin{eqnarray*}
&  &d(u^M,e_i)^2_\mathbb{W}+2(u^M,e_i)_\mathbb{W}(v^M,e_i)_\mathbb{W}dt\\
&=&
2 (u^M,e_i)_\mathbb{W}(\widetilde{G}(u^M),e_i)_{\mathbb{W}}dW(t)
+
|(\widetilde{G}(u^M),e_i)_{\mathbb{W}}|^2_{\mathcal{L}_2(U,\mathbb{R})}dt,
\end{eqnarray*}
and hence
\begin{eqnarray*}
&  &d\|u^M\|^2_\mathbb{W}+2(v^M,u^M)_\mathbb{W}dt\\
&=&
2 (\widetilde{G}(u^M),u^M)_{\mathbb{W}}dW(t)
+
\sum_{i=1}^M|(\widetilde{G}(u^M),e_i)_{\mathbb{W}}|^2_{\mathcal{L}_2(U,\mathbb{R})}dt.
\end{eqnarray*}
By (\ref{W}) we rewrite the above
equation as follows
\begin{eqnarray*}
&  &d[\|u^M\|^2_\mathbb{V}+\|u^M\|^2_*]+2\Big[(v^M,u^M)_\mathbb{V}+\Big(curl(u^M-\alpha \triangle u^M), curl(v^M-\alpha \triangle v^M)\Big)\Big]dt\\
&=&
2 (\widetilde{G}(u^M),u^M)_{\mathbb{V}}dW(t)
+
\sum_{i=1}^M|(\widetilde{G}(u^M),e_i)_{\mathbb{W}}|^2_{\mathcal{L}_2(U,\mathbb{R})}dt\\
& &+
2\Big(curl(u^M-\alpha \triangle u^M), curl(\widetilde{G}(u^M)-\alpha \triangle \widetilde{G}(u^M))\Big)dW(t).
\end{eqnarray*}
Using the definition of $v^M$ and $\widetilde{G}$ (see (\ref{eq:G})), we obtain
\begin{eqnarray*}
&  &d[\|u^M\|^2_\mathbb{V}+\|u^M\|^2_*]+2\Big[(\Psi(u^M),u^M)+\Big(curl(u^M-\alpha \triangle u^M), curl(\Psi(u^M))\Big)\Big]dt\\
&=&
2 (\phi(u^M),u^M)dW(t)
+
\sum_{i=1}^M\lambda_i^2|(\phi(u^M),e_i)|^2_{\mathcal{L}_2(U,\mathbb{R})}dt\\
& &+
2 \Big(curl(u^M-\alpha \triangle u^M), curl(\phi(u^M))\Big)dW(t).
\end{eqnarray*}
Subtracting (\ref{Eq G ito}) from the above equation, we obtain
\begin{eqnarray}\label{eq 04}
&  &d\|u^M\|^2_*+2\Big(curl(u^M-\alpha \triangle u^M), curl(\Psi(u^M))\Big)dt\\
&=&
\sum_{i=1}^M(\lambda_i^2-\lambda_i)|(\phi(u^M),e_i)|^2_{\mathcal{L}_2(U,\mathbb{R})}dt\nonumber\\
&&+
2 \Big(curl(u^M-\alpha \triangle u^M), curl(\phi(u^M))\Big)dW(t),\nonumber
\end{eqnarray}
here we used the fact $2(\Psi(u^M),u^M)=2\nu\|u^M\|^2$.

Since
$$
curl\Big(curl(u^M-\alpha\triangle u^M)\times u^M\Big)
=
(u^M\cdot \nabla)(curl(u^M-\alpha\triangle u^M)),
$$
we have $\Big(curl(u^M-\alpha \triangle u^M),curl\Big(curl(u^M-\alpha\triangle u^M)\times u^M\Big)\Big)=0$.
Hence
\begin{eqnarray*}
& &\Big(curl(u^M-\alpha \triangle u^M), curl(\Psi(u^M))\Big)\\
&=&
 \Big(curl(u^M-\alpha \triangle u^M),curl(-\nu \triangle u^M)\Big)
 \\
&=&
\frac{\nu}{\alpha}\|u^M\|^2_*-\frac{\nu}{\alpha}\Big(curl(u^M-\alpha \triangle u^M),curl\ u^M\Big).
%&&-
%\Big(curl(u^M-\alpha \triangle u^M),F(u^M,t)+G(u^M)\dot{h^\e}(t)\Big).
\end{eqnarray*}
It follows from (\ref{eq 04}) that
\begin{eqnarray}\label{eq 000}
&  &d\|u^M\|^2_*+\frac{2\nu}{\alpha}\|u^M\|^2_*dt-\frac{2\nu}{\alpha}\Big(curl(u^M-\alpha \triangle u^M),curl\ u^M\Big)dt\nonumber\\
%& &-2\Big(curl(u^M-\alpha \triangle u^M),F(u^M,t)+G(u^M,t)\dot{h^\e}(t)\Big)dt\nonumber\\
&=&
\sum_{i=1}^M(\lambda_i^2-\lambda_i)|(\phi(u^M),e_i)|^2_{\mathcal{L}_2(U,\mathbb{R})}dt\nonumber\\
&&+
2 \Big(curl(u^M-\alpha \triangle u^M), curl(\phi(u^M))\Big)dW(t).
\end{eqnarray}

Using the fact that
\begin{eqnarray}\label{Eq 9}
|curl(u)|^2\leq \frac{2}{\alpha}\|u\|^2_\mathbb{V}\ \ \ \text{for any }u\in\mathbb{W},
\end{eqnarray}
we have
\begin{eqnarray}\label{Eq 001}
&  &\Big|\Big(curl(u^M-\alpha \triangle u^M),curl\ u^M\Big)\Big|\nonumber\\
&\leq&
\sqrt{\frac{2}{\alpha}}\|u^M(s)\|_\mathbb{V}\|u^M(s)\|_*\nonumber\\
&\leq&
\frac{1}{2}\|u^M(s)\|^2_*ds+\frac{2}{\alpha}\|u^M(s)\|^2_\mathbb{V}.
\end{eqnarray}
Using similar arguments as that for (\ref{Eq ito 04}), we have
\begin{eqnarray}\label{Eq 004}
\sum_{i=1}^M(\lambda_i^2+\lambda_i)|(\phi(u^M(s)),e_i)|^2_{\mathcal{L}_2(U,\mathbb{R})}
\leq
(1+\frac{1}{\lambda_1})K^2R.
\end{eqnarray}

Combining (\ref{eq 000}), (\ref{Eq 001}) and (\ref{Eq 004}), we arrive at

\begin{eqnarray}\label{eq 005}
d\|u^M\|^2_*+\frac{\nu}{\alpha}\|u^M\|^2_*dt
&\leq&
\frac{4\nu}{\alpha^2}\|u^M\|^2_\mathbb{V}dt+(1+\frac{1}{\lambda_1})K^2Rdt\nonumber\\
&&+
2 \Big(curl(u^M-\alpha \triangle u^M), curl(\phi(u^M))\Big)dW(t).
\end{eqnarray}

By (\ref{Eq 01}), (\ref{Eq lem 01}) and (\ref{eq 005}), we obtain
%$(1+\frac{2(\mathcal{P}^2+\alpha)}{\alpha^2})$(\ref{Eq lem 01})+(\ref{eq 005})
\begin{eqnarray}\label{eq 006}
d\|u^M\|^2_\mathbb{W}+l\|u^M\|^2_\mathbb{W}dt
&\leq&
l_0Rdt
+
(2+\frac{4(\mathcal{P}^2+\alpha)}{\alpha^2})(\phi(u^M),u^M)dW(t)\nonumber\\
& &+
2 \Big(curl(u^M-\alpha \triangle u^M), curl(\phi(u^M))\Big)dW(t),
\end{eqnarray}

here $l=\frac{\nu}{\mathcal{P}^2+\alpha}$, $l_0=(1+\frac{2}{\lambda_1}+\frac{2(\mathcal{P}^2+\alpha)}{\lambda_1\alpha^2})K^2$.

Applying chain rule to $e^{lt}\|u^M\|^2_\mathbb{W}$ and taking the expectation, we obtain
\begin{eqnarray*}
E\|u^M(t)\|^2_{\mathbb{W}}
\leq
e^{-lt}\|u^M(0)\|^2_\mathbb{W}+\frac{l_0R}{l},
\end{eqnarray*}
which is the desired inequality (\ref{eq 01 lem 01}).

Let $\beta>0$ and $\tau$ be a stopping time. Applying ${\rm It\hat{o}}$'s formula to $e^{-\beta t}\|u^M(t)\|^2_\mathbb{W}$, we have
\begin{eqnarray}
& &d(e^{-\beta t}\|u^M\|^2_\mathbb{W})+e^{-\beta t}(\beta+l)\|u^M\|^2_\mathbb{W}dt\nonumber\\
&\leq&
e^{-\beta t}l_0Rdt
+
(2+\frac{4(\mathcal{P}^2+\alpha)}{\alpha^2})e^{-\beta t}(\phi(u^M),u^M)dW(t)\nonumber\\
& &+
2e^{-\beta t} \Big(curl(u^M-\alpha \triangle u^M), curl(\phi(u^M))\Big)dW(t).
\end{eqnarray}
This implies that, for any $n\in\mathbb{N}$,
\begin{eqnarray}
E\Big(e^{-\beta(\tau\wedge n)}\|u^M(\tau\wedge n)\|^2_\mathbb{W}\Big)
\leq
\|u^M(0)\|^2_\mathbb{W}+\frac{l_0R}{\beta}.
\end{eqnarray}
Letting $n\rightarrow \infty$, we obtain (\ref{eq 02 lem 01}).

The proof is complete.
\end{proof}

\vskip 0.3cm
Denote by $u(\cdot, W, u_0)$ the unique solution of (\ref{01}) with initial value $u_0$. By similar arguments as in the proof of
Theorem 4.2 in \cite{RS-12}, we have the following
lemma which will be used later.
\begin{lem}\label{lem 03}
The sequence of Galerkin approximations $(u^M)_{M\geq1}$ satisfies
$$
\lim_{M\rightarrow\infty}E(\|u^M(t)-u(t)\|^2_\mathbb{V})=0,\ t>0,
$$
$$
\lim_{M\rightarrow\infty}E(\int_0^T\|u^M(t)-u(t)\|^2_\mathbb{V}dt)=0.
$$
\end{lem}

 By Lemma \ref{Lem 01}, we deduce that
$$
u^M(t)\rightarrow u(t, W, u_0)\ \text{weakly in}\ L^2(\Omega,\mathcal{F},P; \mathbb{W}).
$$
Furthermore, the following result holds.
\begin{prp}\label{Prop 01}
Assume {\bf (H0)} holds. There exists constant $C_1$ only depending on $\nu,\ \alpha,\ R$ and $\mathcal{O}$ such that
, for any $t\geq0$,
\begin{eqnarray}\label{eq 01 prop 01}
E(\|u(t,W,u_0)\|^2_\mathbb{W})\leq e^{-\frac{\nu t}{\mathcal{P}^2+\alpha}}\|u_0\|^2_\mathbb{W}+C_1.
\end{eqnarray}
Moreover, for any $\beta>0$, there exits a constant $C_\beta^1$ such that for any stopping time $\tau$,
\begin{eqnarray}\label{eq 02 prop 01}
E(e^{-\beta\tau}\|u(\tau,W,u_0)\|^2_\mathbb{W}I_{\tau<\infty})\leq \|u_0\|^2_\mathbb{W}+C^1_\beta.
\end{eqnarray}

\end{prp}

\vskip 0.3cm

Recall $l=\frac{\nu}{\mathcal{P}^2+\alpha}$ and $l_0=(1+\frac{2}{\lambda_1}+\frac{2(\mathcal{P}^2+\alpha)}{\lambda_1\alpha^2})K^2$. Define
the energy functional
$$
E_f(t)=\|f(t)\|^2_\mathbb{W}+\frac{l}{2}\int^t_0\|f(s)\|^2_\mathbb{W}ds.
$$
\begin{lem}\label{Lem G 02}
Assume (H0) holds. There exists $\kappa>0$ such that for any $\kappa_0\leq\kappa/2$
\begin{eqnarray}\label{001}
E\Big[\exp\Big(\kappa_0\sup_{t\geq0}(E_{u^M(\cdot,W,u_0^M)}(t)-l_0Rt)\Big)\Big]
\leq
2\exp(\kappa_0\|u_0^M(0)\|^2_\mathbb{W}).
\end{eqnarray}
\end{lem}
\begin{proof}
Set
$$
\mathcal{A}(t)=\int_0^t(2+\frac{4(\mathcal{P}^2+\alpha)}{\alpha^2})(\phi(u^M),u^M)dW(s)
+
2 \int_0^t\Big(curl(u^M-\alpha \triangle u^M), curl(\phi(u^M))\Big)dW(s).
$$
It is easy to see that
\begin{eqnarray*}
d\langle \mathcal{A}\rangle(t)
\leq
cR\|u^M\|^2_\mathbb{W}dt.
\end{eqnarray*}

Let $\kappa=\frac{l}{cR}$ and
\begin{eqnarray*}
\mathcal{A}_\kappa(t)=\mathcal{A}(t)-\frac{\kappa}{2}\langle \mathcal{A}\rangle(t).
\end{eqnarray*}
By (\ref{eq 006}), we have
\begin{eqnarray}\label{008}
E_{u^M}(t)
\leq
\|u^M(0)\|^2_\mathbb{W}+l_0Rt+\mathcal{A}_\kappa(t).
\end{eqnarray}

Since $e^{\kappa\mathcal{A}_\kappa}$ is a positive supermartingale whose value is 1 at $t=0$, we have
\begin{eqnarray}
P\Big(\sup_{t\geq0}\mathcal{A}_\kappa(t)\geq\rho\Big)
\leq
P\Big(\sup_{t\geq0}\exp(\kappa\mathcal{A}_\kappa(t))\geq \exp(\kappa\rho)\Big)
\leq
\exp(-\kappa\rho),
\end{eqnarray}
which implies, letting $\kappa_0\leq\kappa/2$,
\begin{eqnarray}\label{009}
E\big(e^{\kappa_0\sup\mathcal{A}_\kappa}\Big)
=
1+\kappa_0\int_0^\infty e^{\kappa_0y}P\Big(\sup_{t\geq0}\mathcal{A}_\kappa(t)\geq y\Big)dy
\leq
2.
\end{eqnarray}

Combining (\ref{009}) and (\ref{008}), we deduce (\ref{001}).

\end{proof}

%Consider the following SPDE
%\begin{eqnarray}\label{Eq app}
%d(\tilde{u}-\alpha\tilde{u})+\nu A\tilde{u}dt+curl(\tilde{u}-\alpha\tilde{u})\times \tilde{u}dt
%+
%\rho P_N(\tilde{u}-u(t,W,u_0))dt
%=
%\phi(\tilde{u})dW
%\end{eqnarray}
%with initial condition $\tilde{u}(0)=\tilde{u}_0\in\mathbb{W}$.
Setting $\rho>0$. Let $\tilde{u}^M\in\mathbb{W}_M$ be the solution of the following SPDE
\begin{eqnarray}\label{eq app 01}
& &d(\tilde{u}^M,e_i)_\mathbb{V}+\nu((\tilde{u}^M,e_i))dt+b(\tilde{u}^M,\tilde{u}^M,e_i)dt-\alpha b(\tilde{u}^M,\triangle \tilde{u}^M,e_i)dt+\alpha b(e_i,\triangle \tilde{u}^M,\tilde{u}^M)dt\nonumber\\
&=& (\rho P_N(\tilde{u}^M-u^M(t,W,u_0)),e_i)_{\mathbb{V}}dt
 +(\phi(\tilde{u}^M),e_i)dW(t)
\end{eqnarray}
for any $i\in\{1,2,\cdots,M\}$, with initial value $\tilde{u}^M(0)=P_M\tilde{u}_0$, $\tilde{u}_0\in W$.

\begin{lem}\label{lem 02}

Assume that {\bf (H0)} and {\bf (H1)} hold. There exist $\varpi>0$ and $\kappa_0\geq 0$ such that
\begin{eqnarray}\label{eq: 03}
& &\sup_{t\geq 0}E\Big(\Big(e^t \|r_M(t)\|^2_\mathbb{V}+\int_0^te^s \|r_M(s)\|^2_\mathbb{V}ds\Big)^\varpi\Big)\nonumber\\
&\leq&
2\|r_M(0)\|^{2\varpi}_\mathbb{V}\exp(\kappa_0\|u^M(0)\|^2_\mathbb{W})
\end{eqnarray}
where $r_M=\tilde{u}^M-u^M$.
\end{lem}

\begin{proof}
Note that
\begin{eqnarray*}
d(r_M,e_i)_{\mathbb{V}}+\nu((r_M,e_i))dt+(\delta B,e_i)dt+\rho (P_N r_M,e_i)_{\mathbb{V}}dt
=
(\delta \phi,e_i)dW(s),
\end{eqnarray*}
here $\delta B=B(\tilde{u}^M,\tilde{u}^M)-B(u^M,u^M)$ and $\delta \phi=\phi(\tilde{u}^M)-\phi(u^M)$.

Applying ${\rm It\hat{o}}$'s formula to $(r_M,e_i)^2_{\mathbb{V}}$, and remenbering that $\|r_M\|^2_\mathbb{V}=\sum_{i=1}^M \lambda_i(r_M,e_i)^2_{\mathbb{V}}$,
we have
\begin{eqnarray*}
&&d\|r_M\|^2_\mathbb{V}+2\nu\|r_M\|^2dt+2(\delta B,r_M)dt+2(\rho P_N r_M,r_M)_{\mathbb{V}}dt\\
&=&
2(\delta \phi,r_M)dW(t)+\sum_{i=1}^M\lambda_i|(\delta \phi,e_i)|^2_{\mathcal{L}_2(U,\mathbb{R})}dt.
\end{eqnarray*}

By (\ref{Ineq B 02}),
$$
|2(\delta B,r_M)|=|2(B(r_M,r_M,)u^M)|
\leq
2\Theta\|r_M\|^2_\mathbb{V}\|u^M\|_\mathbb{W}
\leq
\|r_M\|^2_\mathbb{V}+\Theta^2\|r_M\|^2_\mathbb{V}\|u^M\|^2_\mathbb{W}.
$$
By the similar arguments as that for the proof of (\ref{Eq ito 04}),
$$
\sum_{i=1}^M\lambda_i|(\delta \phi,e_i)|^2_{\mathcal{L}_2(U,\mathbb{R})}
\leq
\frac{K^2}{\lambda_1}|\delta \phi|^2_{\mathcal{L}_2(U,\mathbb{V})}
\leq
\frac{K^2L^2_\phi}{\lambda_1}\|r_M\|^2_\mathbb{V}.
$$
Set $l_1=\frac{2\nu}{\mathcal{P}^2+\alpha}-1-\frac{K^2L^2_\phi}{\lambda_1}>0$ and setting $\Lambda_1=\Theta^2$, we have
\begin{eqnarray*}
d\|r_M\|^2_\mathbb{V}+\Big(l_1-\Lambda_1\|u^M\|^2_\mathbb{W}\Big)\|r_M\|^2_\mathbb{V}dt
\leq
2(\delta \phi,r_M)dW(t).
\end{eqnarray*}
Set $G_1(t)=e^{-l_1 t+\Lambda_1 \int_0^t\|u^M(s)\|^2_\mathbb{W}ds}$. By the chain rule, we have
\begin{eqnarray*}
d(e^t G^{-1}_1(t)\|r_M\|^2_\mathbb{V})+e^t G^{-1}_1(t)\|r_M\|^2_\mathbb{V}dt
\leq
2e^t G^{-1}_1(t)(\delta \phi,r_M)dW(t).
\end{eqnarray*}
Integrating the above inequality and taking expectation, it follows that
\begin{eqnarray}\label{1}
E\Big(e^t G^{-1}_1(t)\|r_M(t)\|^2_\mathbb{V}+\int_0^te^s G^{-1}_1(s)\|r_M(s)\|^2_\mathbb{V}ds\Big)
\leq
\|r_M(0)\|^2_\mathbb{V}.
\end{eqnarray}

By $\rm H\ddot{o}lder$ inequality,
\begin{eqnarray}\label{2}
& &E\Big(\Big(e^t \|r_M(t)\|^2_\mathbb{V}+\int_0^te^s \|r_M(s)\|^2_\mathbb{V}ds\Big)^\varpi\Big)\nonumber\\
&\leq&
\sqrt{E(\sup_{t\geq0}G^{2\varpi}_1(t))}
\Big(E\Big(e^t G^{-1}_1(t)\|r_M(t)\|^2_\mathbb{V}+\int_0^te^s G^{-1}_1(s)\|r_M(s)\|^2_\mathbb{V}ds\Big)\Big)^\varpi.
\end{eqnarray}
Choosing $\varpi>0$ sufficiently small, it follows from Lemma \ref{Lem G 02} and condition \textbf{(H1)} that
\begin{eqnarray}\label{3}
E(\sup_{t\geq0}G^{2\varpi}_1(t))
&=&
E(\sup_{t\geq0}e^{-2\varpi l_1 t+2\Lambda_1\varpi \int_0^t\|u^M(s)\|^2_\mathbb{W}ds})\nonumber\\
&\leq&
E(\sup_{t\geq0}e^{\frac{4\varpi \Lambda_1}{l}(E_{u^M}(t)-\frac{l_1l}{2\Lambda_1}t)})\nonumber\\
&\leq&
E\Big(\exp\Big(\sup_{t\geq0}\frac{4\varpi \Lambda_1}{l}(E_{u^M}(t)-\frac{l_1l}{2\Lambda_1}t)\Big)\Big).
\end{eqnarray}

(\ref{1}) (\ref{2}) and (\ref{3}) imply (\ref{eq: 03}).
The proof  is complete.
\end{proof}
\vskip 0.3cm

Recall $X^x$ defined by (\ref{X}). By Theorem 3.4 and Theorem 4.1 in \cite{RS-12}, there exits a unique solution $\tilde{u}=\tilde{u}(\cdot,W,u_0,\tilde{u}_0)$ satisfying
\begin{eqnarray}\label{eq u tilde}
&  &d(\tilde{u} -\alpha \triangle \tilde{u} )+\Big(-\nu \triangle \tilde{u} +curl(\tilde{u} -\alpha \triangle \tilde{u} )\times \tilde{u} \Big)dt
+\rho P_N((\tilde{u}-X^{u_0}) -\alpha \triangle (\tilde{u}-X^{u_0}))dt\nonumber\\
& &=
\phi(\tilde{u})dW,
\end{eqnarray}
with initial value $\tilde{u}(0)=\tilde{u}_0\in\mathbb{W}$.

By Lemma \ref{lem 03} and using similar arguments as Theorem 4.2 in \cite{RS-12}, we have
\begin{lem}\label{lem 04}
The sequence of approximations $(\tilde{u}^M)_{M\geq1}$ satisfies
$$
\lim_{M\rightarrow\infty}E(\|\tilde{u}^M(t)-\tilde{u}(t)\|^2_\mathbb{V})=0,
$$
$$
\lim_{M\rightarrow\infty}E(\int_0^T\|\tilde{u}^M(t)-\tilde{u}(t)\|^2_\mathbb{V}dt)=0.
$$
\end{lem}

Combining Lemma \ref{lem 03}, Lemma \ref{lem 04} and Lemma \ref{lem 02} and applying Fatou's lemma, we have
\begin{prp}\label{prop 02} Assume that {\bf (H0)} and {\bf (H1)} hold. There exist $\varpi>0$ and $\kappa_0>0$ such that
\begin{eqnarray}\label{eq:04}
& &E\Big(\Big(e^t \|r(t)\|^2_\mathbb{V}+\int_0^te^s \|r(s)\|^2_\mathbb{V}ds\Big)^\varpi\Big)\nonumber\\
&\leq&
2\|r(0)\|^{2\varpi}_\mathbb{V}\exp(\kappa_0\|u_0\|^2_\mathbb{W})
\end{eqnarray}
where $r(t)=\tilde{u}(t)-u(t)$.

\end{prp}

%Assume that \textbf{H0} holds. There exist $\varpi>0$, $C>0$ such that
%\begin{eqnarray*}
%& &E\Big(\Big(e^t \|r(t)\|^2_\mathbb{V}+\int_0^te^s \|r(s)\|^2_\mathbb{V}ds\Big)^\varpi\Big)\\
%&\leq&
%2\|r(0)\|^2_\mathbb{V}\exp(\kappa_0\|u_0\|^2_\mathbb{W})
%\end{eqnarray*}
%where $r(t)=\tilde{u}(t)-u(t)$.

\subsection{Completion of the proof}

Now we verify the assumptions \textbf{(A)-(D)} of Theorem \ref{Thm biaozhun}. Denote the solution of (\ref{eq u tilde}) by $\tilde{u}(t,W,u_0,\tilde{u}_0)$.
Set \begin{eqnarray}\label{eq tilde X}
\widetilde{X}(t,W,x_0,\tilde{x}_0)=
\left\{
 \begin{array}{ll}
 & \hbox{$\tilde{u}(t,W,x_0,\tilde{x}_0),\ \ \ \tilde{x}_0\in\mathbb{W}$;} \\
 & \hbox{$\tilde{x}_0,\ \ \ \tilde{x}_0\in\mathbb{V}/\mathbb{W}$,}
 \end{array}
\right.
\end{eqnarray}
and recall (\ref{X}). We set $(\Upsilon(t),\widetilde{\Upsilon}(t))=(X(t,W,x_0),\widetilde{X}(t,W,x_0,\tilde{x}_0))$. Then
$\textbf{(A)}$ is a consequence of the well-posedness of the equations.

We set $\mathcal{H}=\|\cdot\|^2_\mathbb{W}$. By Proposition \ref{Prop 01}, we have $\textbf{(B)}$.
Let $h(u_0,u_1)=-g(u_1)\rho P_N(u_1-u_0)$ and recall \textbf{(H1)}, then $\textbf{(C)}$ holds.
\vskip 0.3cm

By \textbf{(H1)}, we have
\begin{enumerate}
  \item \begin{eqnarray*}
& &P\Big(|\Upsilon(t,W_2,x_0^2)-\Upsilon(t,W_1,x_0^1)|_{\mathbb{V}}\geq Ce^{-{\gamma_0} t}, \widetilde{\Upsilon}(\cdot,W_1,x_0^1,x_0^2)=\Upsilon(\cdot,W_2,x_0^2)\ on\ [0,t]\Big)\\
&&=
P\big(\|r(t)\|_\mathbb{V}\geq Ce^{-{\gamma_0} t}\Big);
\end{eqnarray*}

  \item for any $t_0\geq 0$ and any stopping time $\tau\geq t_0$,
\begin{eqnarray*}
&&P\Big(\int_{t_0}^\tau|h(t)|^2_Udt\geq Ce^{-{\gamma_0} t_0}\ and\ \widetilde{\Upsilon}(\cdot,W_1,x_0^1,x_0^2)=\Upsilon(\cdot,W_2,x_0^2)\ on\ [0,\tau]\Big)\\
&&
\leq
P\Big(\int_{t_0}^\tau\|r(t)\|_\mathbb{V}^2dt\geq \widetilde{C}e^{-{\gamma_0} t_0}\Big);
\end{eqnarray*}

  \item $$
P\Big(\int_{0}^{+\infty}|h(t)|^2_Udt\leq C\Big)\geq P\Big(\int_{0}^{+\infty}\|r(t)\|^2_\mathbb{V}dt\leq \widetilde{C}\Big).
$$
\end{enumerate}
By Chebyshev inequality, we deduce $\textbf{(D)}$ directly from Proposition \ref{prop 02}.

Thus, we can apply Theorem \ref{Thm biaozhun} to obtain
our main result Theorem \ref{thm main}.

\vskip0.5cm {\small {\bf  Acknowledgements}\ \  This work was supported by National Natural Science Foundation of China (NSFC) (No. 11431014, No. 11301498, No. 11401557)}, and the Fundamental Research Funds for the Central Universities (No. WK 0010000033).

%%%%%%%%%%%%%%%%%%%%%%%%%%%%%%%%%%%%%%%%%%%%%%%%%%%%%%%%%%%%%%%%%%%%%%%%%%%%%%%%%%%%%%%%%%%%%%%%%%%%%%%%%%%%%%%

\def\refname{ References}

\end{document}